\documentclass[11pt,a4paper]{article}

\usepackage[utf8]{inputenc}
\usepackage[T1]{fontenc}

\usepackage{amsmath}
\usepackage{amsthm}
\usepackage{amssymb}
\usepackage{cite}
\usepackage{color}
\usepackage{graphicx}
\usepackage{wasysym}

\newtheorem{definition}{Definition}
\newtheorem{theorem}{Theorem}

\begin{document}

\title{\bf A Non-Self-Referential Paradox\\
	in Epistemic Game Theory}

\author{{\sc Ahmad Karimi}\\[1em]
  Department of Mathematics\\
  Behbahan Khatam Alanbia University of Technology\\
  P.O.Box 61635--151, Behbahan, Iran\\
  karimi@bkatu.ac.ir
  }

\maketitle

\begin{abstract}
	 In game theory, the notion of a player's beliefs about the game players' beliefs about other players' beliefs arises naturally. In this paper, we present a non-self-referential paradox in epistemic game theory  which shows that completely modeling  players' epistemic beliefs and assumptions is impossible. Furthermore, we introduce an interactive temporal assumption logic to give an appropriate formalization of the new paradox. Formalizing  the new paradox in this logic shows that there is no complete  interactive temporal assumption  model.
	
	\bigskip
	
	
\end{abstract}

\section{Introduction}\label{intro}

In 2006  Adam Brandenburger and H. Jerome Keisler presented a two-person and self-referential paradox
in epistemic game theory \cite{brandenburger2006impossibility}. They show that the following configuration of beliefs can not be represented: {\it
	Ann believes that Bob assumes that Ann believes that Bob's assumption is wrong}. The Brandenburger-Keisler paradox (BK paradox)  shows impossibility of completely modeling   players' epistemic beliefs and assumptions.
Brandenburger and Keisler present a modal logic interpretation of the paradox \cite{brandenburger2006impossibility}. They introduce two modal operators intended to represent the agents' beliefs
and assumptions. In \cite{pacuit2007understanding}  Pacuit approaches
the BK paradox from a modal logical perspective and presents a
detailed investigation of the paradox in neighborhood models and   hybrid systems. In particular,
he shows that the paradox can be seen as a theorem of an appropriate hybrid logic.

The BK paradox is essentially a self-referential paradox and similarly to any other paradox of
the same kind can be analyzed from a category theoretical or algebraic point of view. In \cite{abramsky2015lawvere}  Abramsky and Svesper analyze the BK paradox  in categorical context. They  present the   paradox as a fixed-point theorem, which can be carried out in any regular category, and show how it can be reduced to a relational form of the one-person
diagonal argument due to Lawvere \cite{lawvere1969diagonal} where he gave a simple form of the (one-person) diagonal argument as a fixed-point lemma in a very general setting.

Ba\c{s}kent  in \cite{bacskent2015some}  approaches the BK paradox from two different
perspectives: non-well-founded set
theory and paraconsistent logic. He shows that the paradox persists in both frameworks for category-theoretical reasons, but with different properties. Ba\c{s}kent makes the connection between self-referentiality and paraconsistency.

On the other hand, Stephen Yablo \cite{yablo1993paradox} introduced a logical paradox in 1993 that is similar to the liar paradox where he used an infinite sequence of statements. Every statement  in the sequence refers to the truth values of the later statements. Therefore, it seems this paradox avoids self--reference. In this paper, we  present a non-self-referential  multi-agent paradox in epistemic game theory which we call   ``Yablo--like Brandenburger-Keisler paradox''.


In an interactive temporal assumption logic we study  interactive reasoning: how agents  reason about the beliefs of other agents over time. In this paper, we introduce an interactive temporal assumption logic to give an appropriate formalization of the  Yablo--like Brandenburger-Keisler paradox.

\section{Brandenburger-Keisler Paradox}\label{sect2}
Brandenburger and Keisler introduced the following two person Russell-style paradox \cite{bacskent2015some,brandenburger2006impossibility,pacuit2007understanding}. Beliefs and assumptions are two main concepts involved in the statement of the paradox.  An assumption is the strongest belief.  Suppose there are two players, Ann and Bob, and consider the following description of beliefs:

\vspace{5mm}

\centerline{Ann believes that Bob assumes that Ann believes}
\centerline{that Bob's assumption is wrong.}

\vspace{5mm}

\noindent A paradox arises when one asks the question
``Does Ann believe that Bob's assumption is wrong?''
Suppose that answer to the above question is ``yes''. Then according to Ann, Bob's assumption is wrong. But, according to Ann, Bob's assumption is Ann
believes that Bob's assumption is wrong. However, since the answer to the above question is ``yes'', Ann believes that this assumption is correct. So Ann
does not believe that Bob's assumption is wrong. Therefore, the answer to
the above question must be ``no''. Thus, it is not the case that Ann believes
that Bob's assumption is wrong. Hence Ann believes Bob's assumption is
correct. That is, it is correct that Ann believes that Bob's assumption is
wrong. So, the answer must have been ``yes''. This is a contradiction!

Brandenburger and Keisler \cite{brandenburger2006impossibility} introduce belief models for two players to present  their impossibility results. In a belief model, each player has a set of states, and each state for one player has beliefs about the states of the other player. We assume the reader is familiar with the belief models and formulation of BK paradox in modal logic, but for the sake of accessibility, we list the main notations, definitions and theorems which will be referred to later on. For more details, we refer the reader to \cite{brandenburger2006impossibility, pacuit2007understanding}.

\begin{definition}
	A {\bf belief model} is a two-sorted structure
	$${\cal{M}}= (U^a, U^b, P^a, P^b, ...)$$
	where $U^a$ and $U^b$ are the nonempty universe sets (for the two sorts), $P^a$ is
	a proper subset of $U^a\times U^b$, $P^b$ is a proper subset of $U^b \times U^a$, and $P^a$, $P^b$
	are serial, that is, the sentences
	$$\forall x\exists  y\;P^a(x, y) \;\;\text{and}\;\; \forall y\exists  x\; P^b(y, x)$$
	hold. To simplify notation,  $x$ is always a variable
	of sort $U^a$ and $y$ is a variable of sort $U^b$. The members of $U^a$ and $U^b$ are called
	{\bf states} for Ann and Bob, respectively. $P^a$ and $P^b$ are called  {\bf possibility relations}. We say $x$ {\bf believes} a set
	$Y \subseteq U^b $ if $\{y :\; P^a(x,y)\}\subseteq Y $, and $x$ {\bf assumes} $Y$ if $\{y : \; P^a(x, y)\} = Y$.
\hfill $\oplus\hspace{-0.77em}\otimes$
\end{definition}
For an arbitrary structure ${\cal{N}}= (U^a, U^b, P^a, P^b, ...)$, by the first order language for
${\cal{N}}$ we mean the two-sorted first order logic with sorts for $U^a$ and $U^b$ and
symbols for the relations in the vocabulary of ${\cal{N}}$. Given a first order formula
$\varphi(u)$ whose only free variable is $u$, the set defined by $\varphi$ in ${\cal{N}}$ is the set
$\{u :\; \varphi(u) \; \text{is true in}\; {\cal{N}}\}$.
In general, by a language for ${\cal{N}}$ we will mean a subset of the set of all formulas of the first order language for ${\cal{N}}$. Given a language ${\cal L}$ for ${\cal N}$, we let
{${\cal L}_a$, ${\cal L}_b$} be the {families of all subsets of $U^a$, $U^b$} respectively which are defined
by formulas in ${\cal L}$. Also, The diagonal formula $D(x)$ is the first order formula
$$\forall y \;(P^a(x,y) \rightarrow \neg P^b(y,x)).$$
This is the formal counterpart to the intuitive statement: ``Ann believes that
	Bob's assumption is wrong.''

\begin{definition}
	Let ${\cal M}$ be a belief model, and let ${\cal L}$ be a language for it. The model
	${\cal M}$ is {\bf complete} for ${\cal L}$ if each nonempty set $Y \in {\cal L}_b$ is assumed by some $x\in U^a$, and each nonempty $X \in {\cal L}_a$ is assumed by some $y\in U^b$.
\hfill $\oplus\hspace{-0.77em}\otimes$
\end{definition}

In their paper \cite{brandenburger2006impossibility}, Brandenburger and Keisler prove that  there is no belief model which is complete for a language ${\cal L}$ which
contains the tautologically true formulas, $D(x)$, $\forall y\;  P^b(y,x)$ and formula built from the formal counterpart of the BK paradox.



They also present modal formulations of the BK paradox  by presenting
a  specific  modal logic, called {\bf interactive assumption logic}, with two modal operators for  ``belief'' and ``assumption''. For semantical interpretations Brandenburger and Keisler use Kripke models \cite{brandenburger2006impossibility}. Pacuit \cite{pacuit2007understanding} presents a
detailed investigation of the BK paradox in neighborhood models and in hybrid systems. He shows that the paradox can be seen as a theorem of an appropriate hybrid logic.

\section{Yablo's Paradox}

To counter a general belief that all the paradoxes stem from a kind
of circularity (or involve some self--reference, or use a diagonal
argument) Stephen Yablo designed a paradox in 1993 that seemingly
avoided self--reference \cite{yablo1993paradox, yablo1985truth}. Since then much debate
has been sparked in the philosophical community as to whether
Yablo's Paradox is really circular--free or involves some
circularity (at least hidden or implicitly); see e.g.
\cite{beall2001yablo, bringsjord2003mental,  bueno2003yablo, bueno2003paradox, ketland2004bueno, ketland2005yablo, priest1997yablo, sorensen1998yablo, yablo2004circularity}. Unlike the liar paradox, which uses a single sentence, this paradox applies an infinite sequence of statements. There is no consistent way to assign truth values to all the statements, although no statement directly refers to itself. Yablo considers the following sequence of sentences $\{S_i\}$:
\begin{align*}
S_1 : \forall k > 1; \;S_k \;\; \text{is untrue},\\
S_2 : \forall k > 2; \;S_k \;\; \text{is untrue},\\
S_3 : \forall k > 3; \;S_k \;\; \text{is untrue},\\
\vdots ~~~~~~~~~~~~~~~~~~~~~~~~~~~~~~~~
\end{align*}
The paradox follows from the following deductions. Suppose $S_1$ is true. Then for any $k > 1$, $S_k$ is not true. Specially, $S_2$ is not true. Also, $S_k$ is not true for any $k > 2$. But this is exactly
what $S_2$ says, hence $S_2$ is true after all. Contradiction! Suppose then that $S_1$ is false. This means that there is a $k > 1$ such that $S_k$ is true. But we can repeat the reasoning, this time with respect to $S_k$ and reach a contradiction again. No
matter whether we assume $S_1$ to be true or false, we reach a contradiction. Hence the paradox.  Yablo's paradox can be viewed as a non-self-referential liar's paradox; it has been used to give
alternative proof for G\"odel's first
incompleteness theorem \cite{cieslinski2013godelizing,leach2014yablifying}. Recently in \cite{karimi2013diagonalizing,karimi2014theoremizing}, formalization of  Yablo's paradox and its different versions in Linear Temporal Logic (LTL) yield genuine  theorems in this logic.

%
%
%

\section{A Non-Self-Referential Paradox \\ in Epistemic Game Theory
}

In this section, we present a non-self-referential paradox in epistemic game theory. Unlike the BK paradox which uses   one single statement on belief and assumption of agents, this paradox consists of two sequences of players (or agents) and a sequence of statements of agent's belief and assumptions.   Indeed, this  paradox  shows that not every configuration of agents' beliefs and assumptions can be represented.   Let us consider two infinite sequence of players $\{A_i\}$ and $\{B_i\}$, with the   following description of beliefs:\\
\vspace{2mm}

\centerline{
	\begin{tabular}{ccc}
		$A_1$ & \hspace{7mm} & $B_1$\\
		$A_2$ & & $B_2$\\
		$A_3$ & & $B_3$\\
		$\vdots$ & & $\vdots$	
	\end{tabular}
}

\vspace{5mm}

\centerline{For all $i$, $A_i$ believes that $B_i$ assumes that}
\centerline{ for all $j>i,\;$ $A_j$ believes that $B_j$'s assumption is wrong. }

\vspace{10mm}

A paradox arises when one asks the question ``Does $A_1$ believe that $B_1$'s assumption is wrong?''\\

Suppose that  the answer to
the above question is  ``no''. Thus, it is not the case that $A_1$ believes
that $B_1$'s assumption is wrong. Hence $A_1$ believes $B_1$'s assumption is
correct. That is, it is correct that for all $j>1,\;$ $A_j$ believes that $B_j$'s assumption is wrong. Specially, $A_2$ believes that $B_2$'s assumption is wrong. On the other hand, since for all $j>2,\;$ $A_j$ believes that $B_j$'s assumption is wrong, one can conclude that $A_2$ believes $B_2$'s assumption is
correct. Therefore, at the same time $A_2$ believes that $B_2$'s assumption is both correct and wrong. This is a contradiction!

If the answer to the above question is ``yes''. Then according to $A_1$, $B_1$'s assumption is wrong. But, according to $A_1$, $B_1$'s assumption is  that ``for all $j>1,\;$ $A_j$ believes that $B_j$'s assumption is wrong''. Thus, there is $k>1$ for which $A_k$ believes that $B_k$'s assumption is correct. Now we can apply the same reasoning we used before about $A_k$ and $B_k$ to reach the contradiction! Hence the paradox. This paradox is a non-self-referential multi-agent version of the BK paradox.


\section{Interactive Temporal Assumption  Logic}

In this section, we introduce an Interactive Temporal Assumption Logic (iTAL)  to present an appropriate formulation of the non-self-referential Yablo-like BK paradox.  The  \textit{interactive temporal assumption language} ${\cal L}_{iTAL}$ contains individual propositional symbols, the propositional
connectives, the linear-time operators $\ocircle,\Box,\lozenge$ and the epistemic operators ``Believe'' and ``Assumption'': for each pair of players $ij$ among Ann and Bob, the operator $B^{ij}$  will be beliefs for player $i$ about $j$, and $A^{ij}$ is the assumption for  $i$ about $j$. In words, $B^{ij}\phi$ means that the agent $i$ believes $\phi$ about $j$, and $A^{ij}\phi$ is that the agent $i$ assumes $\phi$ about agent $j$.
The temporal operators $\circ, \Box,$ and $\lozenge$ are called \textit{next time}, \textit{always (or henceforth)}, and
\textit{sometime (or eventuality)} operators, respectively. Formulas $\ocircle \varphi$, $\Box \varphi$, and $\lozenge \varphi$ are
typically read ``next $\varphi$'', ``always $\varphi$'', and ``sometime $\varphi$''. We note that $\lozenge \varphi \equiv \neg\Box\neg \varphi$.
\begin{definition}
	Formulas in ${\cal L}_{iTAL}$ are defined  as follows:

\qquad $\phi \;\; := p\;|\;\neg \phi \;|\; \phi\wedge\psi \;|\; \ocircle\phi \; |\; \Box \phi\; | \; B^{ij} \phi\;|\; 	A^{ij} \phi$
\hfill $\oplus\hspace{-0.77em}\otimes$
\end{definition}

For semantical interpretations we introduce an appropriate class of Kripke models.

\begin{definition}
An iTAL-Model  is  a Kripke  structure $${\cal W} = (W, \Bbb{N}, \{P_n: n\in \Bbb{N}\},U^a,U^b,V),$$ where $W$ is nonempty set, $\Bbb{N}$ is the set of natural numbers, for each $n\in \Bbb{N}, \; P_n$ is a binary relation $P_n\subseteq W\times W$ and  $U^a,U^b$ are disjoint sets such that $(U^a,U^b,P_n^a,P_n^b)$ is a belief model, where $U^a\cup U^b = W,\; P_n^a = P_n \cap U^a\times U^b$, and $P_n^b= P_n\cap U^b\times U^a$. $V: \text{\bf Prop} \rightarrow 2^{\Bbb{N}\times W}$ is a function mapping to each propositional letter $p$ the subset $V(p)$ of
Cartesian product $\Bbb{N}\times W$. Indeed, $V(p)$ is the set of pairs $(n,w)$ such that $p$
is true in the world $w$ at the moment $n$.
\hfill $\oplus\hspace{-0.77em}\otimes$
\end{definition}
The satisfiability of a formula $\varphi\in {\cal L}_{iTAL}$ in a model ${\cal W}$, at a moment
of time  $n\in \Bbb{N}$ in a world $w\in W$, denoted by ${\cal W}_n^w\Vdash\varphi$ (in short; $(n,w)\Vdash\varphi$ ), is defined
inductively as follows:

\begin{itemize}\itemindent=-1.5em
	\item $(n,w) \Vdash p \Longleftrightarrow (n,w)\in V(p)$ for $p \in $ {\bf Prop},
	\item  $(n,w) \Vdash \neg \varphi \Longleftrightarrow (n,w)\not\Vdash \varphi$,
	\item $(n,w) \Vdash \varphi\wedge\psi \Longleftrightarrow (n,w)\Vdash \varphi$ and $(n,w)\Vdash \psi$,
	\item $(n,w)\Vdash \ocircle\varphi \Longleftrightarrow (n+1,w)\Vdash \varphi$,
	\item $(n,w) \Vdash\Box\varphi \Longleftrightarrow \forall m\geq n\;\; (m,w)\Vdash \varphi$,
\item $(n,w) \Vdash B^{ij}\varphi \Leftrightarrow (n,w)\Vdash {\bf U}^c \wedge \forall z [(P_n(w,z)\wedge (n,z) \Vdash {\bf U}^d) \rightarrow (n,z)\Vdash \varphi])$,
\item $(n,w) \Vdash A^{ij}\varphi \Leftrightarrow (n,w)\Vdash {\bf U}^c \wedge \forall z [(P_n(w,z)\wedge (n,z)\Vdash {\bf U}^d) \leftrightarrow (n,z)\Vdash \varphi])$.
\end{itemize}
Let $x$ has sort $U^a$, $y$ has sort $U^b$ and $\varphi$ is a statement about $y$. Intuitively,  $(n,x) \Vdash B^{ab}\varphi $ says that ``in time $n$, $x$ believes $\varphi(y)$'', and  $(n,x) \Vdash A^{ab}\varphi $ says that ``in time $n$, $x$ assumes $\varphi(y)$''. A formula is {\it valid} for $V$ in ${\cal W}$ if it is true at all $w\in W$, and {\it satisfiable}
for $V$ in ${\cal W}$ if it is true at some $w\in W$. In an interactive assumption  model ${\cal W}$,   we will always suppose
that {\bf D} is a propositional symbol,
and $V$ is a valuation in ${\cal W}$ such that $V({\bf D})$ is the set
$$D = \big \{(n,x) \in W\times \Bbb{N} :\;  (\forall y\in  W)[P_n(x,y)\rightarrow \neg P_n(y,x)]\big\}.$$

In the rest of paper, we present our formulation of the non-self-referential Yablo-like BK paradox in the interactive temporal assumption setting. The thought is that
we can make progress by thinking of the sequences of agents in Yablo-like BK paradox not as infinite families of agents but as a two individual agents that their belief and assumption can be evaluated in
lots of times (temporal states) in an epistemic temporal model. Thus, the  emergence of the Yablo-like BK paradox should be
the same   as the derivability of a particular formula in the epistemic temporal logic.

Let us have a closer look at the Yablo-like BK paradox:
\vspace{5mm}

\centerline{For all $i$, $A_i$ believes that $B_i$ assumes that}
\centerline{ for all $j>i,\;$ $A_j$ believes that $B_j$'s assumption is wrong. }
\vspace{5mm}
Now suppose that there are two players, namely Ann and Bob. Assume that $A_i$ and $B_i$ are the counterparts of Ann and Bob in the $i^{th}$ temporal state. Then infinitely many statements in the Yablo-like BK paradox can be represented in just one single formula using  temporal tools:

$$\Box\big[\text{Ann believes  that Bob assumes that}$$
$$\big(\!\!\ocircle\!\Box (\text{Ann believes  that Bob's assumption is wrong})\big)\big]$$

The interpretation of above formula in English   can be seen as:
``Always it is the case that  Ann believes that Bob assumes from the next time henceforth that Ann believes that  Bob's assumption is wrong''. We note that   Ann and Bob in each state refer to their belief and assumptions in the next temporal states; not the state that they are in.

\begin{theorem}\label{ietltheorem1}
	In an interactive temporal assumption model ${\cal W}$, if $\Box(A^{ab}{\bf U}^b)$ is satisfiable, then
$$\Box\big[B^{ab}A^{ba}(\ocircle\Box {\bf D})\big] \longrightarrow \Box{\bf D} $$
is valid in  ${\cal W}$.	
\end{theorem}
\begin{proof}
	Since $\Box(A^{ab}{\bf U}^b)$ is satisfiable, there exist $k\in\Bbb{N}$ and $\tilde{x}\in W$ for which $(k,\tilde{x})\Vdash\Box(A^{ab}{\bf U}^b)$. This means that $\forall l\geq k\;\; (l,\tilde{x})\Vdash A^{ab}{\bf U}^b$. So, for all $l\geq k$, 	$(l,\tilde{x})\Vdash {\bf U}^a\wedge \forall y\; [P_l^a(\tilde{x},y) \wedge (l,y)\Vdash {\bf U}^b \leftrightarrow (l,y)\Vdash {\bf U}^b]$. Thus
	\begin{align}\label{abc1}
	\forall l\geq k \;\; (l,\tilde{x})\Vdash {\bf U}^a\wedge \forall y\; [P_l^a(\tilde{x},y)].
	\end{align}
	Assume $(k,\tilde{x})\Vdash\Box\big[B^{ab}A^{ba}(\ocircle\Box {\bf D})\big]$;  we show  that $(k,\tilde{x})\Vdash\Box{\bf D}$. To this end, for a moment, suppose that $(k,\tilde{x})\not\Vdash\Box{\bf D}$. So, there is $m\geq k$ for which $(m,\tilde{x})\not\Vdash{\bf D}$. We can   assume that $m=k$. Thus,  $(k,\tilde{x})\Vdash\neg{\bf D}$ which means that $\exists \tilde{y}\;[P_k^a(\tilde{x},\tilde{y})\wedge P_k^b(\tilde{y},\tilde{x})]$. By   (\ref{abc1}) for $l=k$, $(k,\tilde{x})\Vdash {\bf U}^a\wedge \forall y\; [P_k^a(\tilde{x},y)]$. Since $(k,\tilde{x})\Vdash\Box\big[B^{ab}A^{ba}(\ocircle\Box {\bf D})\big]$, then
	\begin{align*}
	 &\forall l\geq k \; (l,\tilde{x})\Vdash B^{ab}A^{ba}(\ocircle\Box {\bf D})\\
	 \Longrightarrow\; & \forall l\geq k \; (l,\tilde{x})\Vdash {\bf U}^a \wedge \forall y [P_l^a(\tilde{x},y) \wedge (l,y)\Vdash {\bf U}^b \rightarrow (l,y)\Vdash A^{ba}(\ocircle\Box {\bf D})]  \\
	 \Longrightarrow\; & \forall l\geq k\; \forall y\;  (l,y)\Vdash A^{ba}(\ocircle\Box {\bf D})\\
	 \Longrightarrow\; & \forall l\geq k\;  \forall y\;   \big((l,y)\Vdash {\bf U}^b \wedge  \forall w\; [P_l^b(y,w) \wedge (l,w)\Vdash {\bf U}^a \leftrightarrow (l,w)\Vdash \ocircle\Box {\bf D}]\big).
	\end{align*}
	Therefore, since $P_k^b(\tilde{y},\tilde{x})$, one can conclude  that $(k,\tilde{x})\Vdash \ocircle\Box {\bf D}$. So, we have  $\forall l\!\geq\!k+1 \; (k,\tilde{x})\Vdash \ocircle\Box {\bf D}$, thus $\forall l\!\geq\!k+1 \; [P_l^a(\tilde{x},\tilde{y}) \rightarrow   \neg P_l^b(\tilde{y},\tilde{x})]$. In particular,  for $k+1$, $P_{k+1}^a(\tilde{x},\tilde{y}) \rightarrow   \neg P_{k+1}^b(\tilde{y},\tilde{x})$. By (\ref{abc1}), since  $P_{k+1}^a(\tilde{x},\tilde{y})$ then $ \neg P_{k+1}^b(\tilde{y},\tilde{x})$. On the other hand, since $(k+1,\tilde{x})\Vdash \ocircle\Box {\bf D}$, then $P_{k+1}^b(\tilde{y},\tilde{x})$ which is a contradiction! Therefore, $(k,\tilde{x})\Vdash\Box{\bf D}$ which shows that
	$$\Box\big[B^{ab}A^{ba}(\ocircle\Box {\bf D})\big] \longrightarrow \Box{\bf D} $$
	is valid in ${\cal W}$.
	\end{proof}

\vspace{5mm}

\begin{theorem}\label{ietltheorem2}
	In any interactive  temporal assumption model ${\cal W}$, the formula
	$$\neg\Box\big[B^{ab}A^{ba}({\bf U}^a\wedge\ocircle\Box {\bf D})\big] $$
	is valid.	
\end{theorem}

\begin{proof}
	If not, then $(k,\tilde{x})\Vdash\Box\big[B^{ab}A^{ba}(\ocircle\Box {\bf D})\big]$ for some sate $(k,\tilde{x})$. By Theorem~\ref{ietltheorem1}, $(k,\tilde{x})\Vdash\Box {\bf D}$ so $\forall l\geq k,\; (l,\tilde{x})\Vdash {\bf D}$, in particular for $l=k$. Thus, $\forall y\; [P_k^a(\tilde{x},y)\rightarrow \neg P_k^b(y,\tilde{x})]$. On the other hand,  $(k,\tilde{x})\Vdash B^{ab}A^{ba}{\bf U^a}$. By the  definitions of $B^{ab}$ and $A^{ba}$ on ${\bf U^a}$, there is some $\tilde{y}$ such that the relation $[P_k^a(\tilde{x},\tilde{y})\wedge P_k^b(\tilde{y},\tilde{x})]$ holds. So, $(k,\tilde{x})\Vdash\neg{\bf D}$ which is a contradiction! Therefore, $\neg\Box\big[B^{ab}A^{ba}({\bf U}^a\wedge\ocircle\Box {\bf D})\big]$ is valid in ${\cal W}$.
\end{proof}

\section{Conclusions}
In game theory, the notion of a player's beliefs about the game 
player's beliefs about other players' beliefs arises naturally. We presented a non-self-referential paradox in epistemic game theory which we called ``Yablo-like Brandenburger-Keisler paradox''. Arising the paradox  shows that completely modeling the  players' epistemic beliefs and assumptions is impossible. We formalized  Yablo-like Brandenburger-Keisler paradox in the interactive  temporal assumption logic  and showed that there is no complete   model for the set of all modal formulas built from ${\bf U}^a, {\bf U}^b$ and ${\bf D}$.

%


%

\end{document}